\documentclass[12pt]{amsart}
\usepackage{fullpage}
\usepackage{times}
\usepackage{latexsym}
\usepackage{amssymb}
\usepackage[all]{xy}

\newtheorem{thm}{Theorem}[section]
\newtheorem{prop}[thm]{Proposition}
\newtheorem{cor}[thm]{Corollary}
\newtheorem{con}[thm]{Conjecture}
\newtheorem{lem}[thm]{Lemma}

\newtheorem{prob}[thm]{Problem}

\theoremstyle{remark}
\newtheorem{rem}[thm]{Remark}

\newtheorem{ex}[thm]{Example}

\usepackage{hyperref}

\newcommand{\N}{\mathbb{N}}

\newcommand{\R}{\mathbb{R}}
\newcommand{\Z}{\mathbb{Z}}

\newcommand{\SO}{\mathrm{\SO}}
\newcommand{\SL}{\mathrm{SL}}

\title[On a problem of Hopf]
{On a problem of Hopf for circle bundles over\\ aspherical manifolds with hyperbolic fundamental groups}

\author{Christoforos Neofytidis}
\address{Department of Mathematics, Ohio State University, Columbus, OH 43210, USA}
\email{neofytidis.1@osu.edu}
\date{\today}

\begin{document}

\begin{abstract}

We prove that a circle bundle over a closed oriented aspherical manifold with hyperbolic fundamental group admits a self-map of absolute degree greater than one if and only if it is virtually trivial. This generalizes in every dimension the case of circle bundles over hyperbolic surfaces, for which the result was known by the
 work of Brooks and Goldman on the Seifert volume. As a consequence, we verify the following strong version of a problem of Hopf  for the above class of manifolds: Every self-map of non-zero degree of a  
circle bundle over a closed oriented aspherical manifold with hyperbolic fundamental group is either homotopic to a homeomorphism or homotopic to a non-trivial covering and the bundle is virtually trivial. 
As another application, we derive the first examples of non-vanishing numerical invariants that are monotone with respect to the mapping degree on non-trivial circle bundles over aspherical manifolds with hyperbolic fundamental groups in any dimension.
\end{abstract}

\maketitle

\section{Introduction}

A long-standing question of Hopf (cf. Problem 5.26 in Kirby's list~\cite{Kirby}) asks the following: 

\begin{prob}{\normalfont (Hopf).}
\label{Hopfprob}
Given a closed oriented manifold $M$, is every self-map $f\colon M\longrightarrow M$ of degree $\pm 1$ a homotopy equivalence?
\end{prob}

A complete solution to Hopf's problem seems to be currently out of reach. Nevertheless, some affirmative answers are known for certain classes of manifolds and dimensions, most notably for simply connected manifolds (by Whitehead's theorem), for manifolds of dimension at most four with Hopfian fundamental groups~\cite{Haus} (recall that a group is called Hopfian if every surjective endomorphism is an isomorphism), and for aspherical manifolds with hyperbolic fundamental groups (e.g. negatively curved manifolds). The latter groups are Hopfian~\cite{Ma,Sela}, thus, the asphericity assumption together with the simple fact that any map of degree $\pm 1$ is $\pi_1$-surjective, answer in the affirmative Problem \ref{Hopfprob} for closed aspherical manifolds with hyperbolic fundamental groups.  

In fact, the assumption about degree $\pm 1$ is unnecessary in answering in the affirmative Problem \ref{Hopfprob} for aspherical manifolds with non-elementary hyperbolic fundamental groups, because 
those manifolds cannot admit self-maps of degree other than $\pm1$ or zero~\cite{BHM,Sela1,Sela,Min,Min1}; cf. Section~\ref{ss:hyperbolicmappings}. Hence, every self-map of non-zero degree of a closed oriented aspherical manifold with non-elementary hyperbolic fundamental group  
is a homotopy equivalence. Of course, the latter statement does not hold
for all (aspherical) manifolds, because, for example, the circle admits self-maps of any degree. Nevertheless, every self-map of the circle of degree greater than one is homotopic to a (non-trivial) covering. The same is true for every self-map of nilpotent manifolds~\cite{Bel} and for certain solvable mapping tori of homeomorphisms of the $n$-dimensional torus~\cite{Wangorder,Neoorder}. In addition, every 
non-zero degree self-map of a $3$-manifold $M$ is either a homotopy equivalence or homotopic to a covering map, unless the fundamental group of each prime summand of $M$ is finite or cyclic~\cite{Wang}. 
The above results suggest the following question for aspherical manifolds:

\begin{prob}[Strong version of Hopf's problem for aspherical manifolds]\label{Hopfstrong}
Is every non-zero degree self-map of a closed oriented aspherical manifold either a homotopy equivalence or homotopic to a non-trivial covering? 
\end{prob}

In dimension three, hyperbolic manifolds and manifolds containing a hyperbolic piece in their JSJ decomposition do not admit any self-map of degree greater than one\footnote{Equivalently, of absolute degree greater than one, by taking $f^2$ whenever $\deg(f)<-1$.} due to the positivity of the simplicial volume~\cite{Gromov}. (Recall that the simplicial volume $\|\cdot\|$ satisfies $\|M'\|\geq|\deg(f)|\cdot\|M\|$ for every map $f\colon M'\longrightarrow M$.) The other classes of aspherical $3$-manifolds which do not admit self-maps of degree greater than one 
are $\widetilde{SL_2}$-manifolds~\cite{BG} and graph manifolds~\cite{DW2}, since those manifolds have another (virtually) positive invariant that is monotone with respect to mapping degrees, namely the Seifert volume (introduced in~\cite{BG} by Brooks and Goldman).
In particular, non-trivial circle bundles over closed 
hyperbolic surfaces (which are modeled on the $\widetilde{SL_2}$ geometry) do not admit self-maps of degree greater than one. 
At the other end, it is clear that trivial circle bundles over (hyperbolic) surfaces, i.e. products $S^1\times \Sigma$, admit self-maps of any degree (and those maps are either homotopy equivalences or homotopic to non-trivial coverings~\cite{Wang}). 

Recall that a circle bundle $M\stackrel{\pi}\longrightarrow N$ is classified by its Euler class $e\in H^2(N;\Z)$; in particular, $M$ is virtually trivial if and only if $e$ is torsion.  
For a circle bundle $M$ over a closed oriented surface $\Sigma$, its Euler class $e\in H^2(\Sigma)=\Z$ is either zero and the bundle is trivial (i.e. $M=S^1\times \Sigma$) or $e$ is not zero and non-torsion and the bundle is not virtually trivial. The main result of this paper is that the non-existence of self-maps of degree greater than one on non-trivial circle bundles over closed oriented hyperbolic surfaces (i.e. over closed oriented aspherical $2$-manifolds with hyperbolic fundamental groups) can be extended in any dimension. In fact, we prove the following stronger statement:

\begin{thm}\label{t:main}
An oriented circle bundle over a closed oriented aspherical manifold with hyperbolic fundamental group admits a self-map of absolute degree greater than one if and only if it is virtually trivial.
\end{thm}

The ``if" direction holds more generally without any assumption on the hyperbolicity of the fundamental group of the base: 

\begin{ex}\label{e:torsion}
Let $M$ be a virtually trivial oriented circle bundle over a closed oriented manifold $N$. Then its Euler class $e\in H^2(N)$ is $k$-torsion for some $k$. Since $M$ is fiberwise oriented, $M$ is a principal $\mathrm{U}(1)$-bundle and thus $M$ can be viewed as the (associated) complex line bundle whose first Chern class is the Euler class $e$. Consider the tensor product $M\otimes\cdots\otimes M$ of $k+1$ copies of $M$. Then the first Chern class of $M\otimes\cdots\otimes M$ is
\[
c_1(M\otimes\cdots\otimes M)=(k+1)c_1(M)=c_1(M),
\]
and so $M\otimes\cdots\otimes M$ is isomorphic to $M$. Taking the $k+1$ power of a section of $M$ gives us a fiberwise map
\[
f\colon M\to M\otimes\cdots\otimes M,
\]
which has degree $k+1$ on the fibers and degree one on the base $N$. Thus $\deg(f)=k+1$.
\end{ex}

\subsection*{Outline of the proof of the main theorem}

In view of Example \ref{e:torsion}, the proof of Theorem \ref{t:main} amounts in showing that if an oriented circle bundle $M$ over a closed oriented aspherical manifold $N$ with $\pi_1(N)$ hyperbolic admits a self-map $f$ of degree greater than one, then $M$ must be virtually trivial. 
We will show that such $f$ is in fact homotopic to a fiberwise non-trivial self-covering of $M$, and thus the powers of $f$ induce a purely decreasing sequence 
\begin{equation}\label{eq.sequenceintro}
\pi_1(M)\supsetneq f_*(\pi_1(M))\supsetneq\cdots\supsetneq f^m_*(\pi_1(M))\supsetneq f^{m+1}_*(\pi_1(M))\supsetneq\cdots.  
\end{equation}
Using this sequence, we will be able to obtain an infinite index subgroup of $\pi_1(M)$ given by
\begin{equation*}
G:=\mathop{\cap}_{m}f^m_*(\pi_1(M)).
\end{equation*}
The last part of the proof uses the concept of groups infinite index presentable by products (IIPP) and characterizations of groups fulfilling this condition~\cite{NeoIIPP}. More precisely, we will see that the multiplication map
\begin{equation*}
\varphi\colon C(\pi_1(M))\times G\longrightarrow \pi_1(M)
\end{equation*}
defines a presentation by products for $\pi_1(M)$, where both $G$ and the center $C(\pi_1(M))$ have infinite index in $\pi_1(M)$. This will lead us to the conclusion that $\pi_1(M)$ has a finite index subgroup isomorphic to a product and $M$ is virtually trivial.

\begin{rem}
In the proof of Theorem \ref{t:main} we will use the fact that the base is an aspherical manifold which does not admit self-maps of degree greater than one, and its fundamental group is Hopfian with trivial center. Thus we can extend Theorem \ref{t:main} (and its consequences, cf. Section \ref{s:applications}) to any circle bundle over a closed oriented manifold $N$ that fullfils the aforementioned properties. For instance, if $N$ is an
 irreducible locally symmetric space of non-compact type, then it is aspherical, it has positive simplicial volume~\cite{LS,Bu} (and thus does not admit self-maps of degree greater than one), and $\pi_1(N)$ is Hopfian~\cite{Ma} without center~\cite{Ra}. 
\end{rem}

\begin{rem}
A decreasing sequence (\ref{eq.sequenceintro}) exists whenever an aspherical manifold $M$ admits a self-map $f$ of degree greater than one and $\pi_1(\overline{M})$ is Hopfian for every finite cover $\overline{M}$ of $M$ (which is conjectured to be true, cf. Section \ref{s:applications}). This gives further evidence towards an affirmative answer to Problem \ref{Hopfstrong}, since the existence of such a sequence is a necessary condition for $f$ to be homotopic to a non-trivial covering.
Now, every finite index subgroup of the fundamental group of a circle bundle over a closed aspherical manifold with hyperbolic fundamental group is indeed Hopfian and therefore this gives us an alternative way of obtaining sequence (\ref{eq.sequenceintro}). We will discuss the Hopf property for those circle bundles  
and Problem \ref{Hopfstrong} more generally in Section \ref{s:Hopfdiscussion}.
\end{rem}

\subsection*{Acknowledgments}

I would like to thank Michelle Bucher, Pierre de la Harpe, Jean-Claude Hausmann, Wolfgang L\"uck, Jason Manning, Dennis Sullivan and Shmuel Weinberger for useful comments and discussions. 
I am especially thankful to Wolfgang L\"uck for suggesting to extend the results of a previous version of this paper to circle bundles over aspherical manifolds with hyperbolic fundamental groups. Also, I am grateful to an anonymous referee for suggesting Example \ref{e:torsion}, which pointed out a mistake in a previous version of this paper.
The support of the Swiss National Science Foundation is gratefully acknowledged.

\section{Applications of the main result}\label{s:applications}

Before proceeding to the proof of Theorem \ref{t:main}, let us mention a few consequences of Theorem \ref{t:main} or of parts of its proof.

It is a long-standing question (motivated by Problem \ref{Hopfprob}) whether the fundamental group of every closed aspherical manifold is Hopfian (see~\cite{Neu} for a discussion). If this is true, then every self-map of an aspherical manifold of degree $\pm 1$ is a homotopy equivalence. 
In the course of the proof of Theorem \ref{t:main}, we will see that every self-map of a circle bundle over a closed oriented aspherical  
manifold with hyperbolic fundamental group is homotopic to a fiberwise covering map, and this alone shows that Problem \ref{Hopfprob} and, in most of the cases, Problem \ref{Hopfstrong} have indeed affirmative answers for self-maps of those manifolds. More interestingly, 
Theorem \ref{t:main} implies the following complete characterization with respect to Problem \ref{Hopfstrong}: 

\begin{cor}\label{c:Hopf}
Every self-map of non-zero degree of an oriented circle bundle over a closed oriented aspherical 
manifold with hyperbolic fundamental group either is a homotopy equivalence or is homotopic to a fiberwise non-trivial covering (and to a non-trivial covering in dimensions different that four and five) and the bundle is virtually trivial.
\end{cor}

\begin{rem}[The Borel conjecture: From homotopy equivalences to homeomorphisms]
\label{r:homeo}
In most cases, an even stronger conclusion holds for the homotopy equivalences of Corollary \ref{c:Hopf}.
Recall that the Borel conjecture asserts that any homotopy equivalence between two closed aspherical manifolds is homotopic to a homeomorphism. (Note that the Borel conjecture does not hold in the smooth category or for non-aspherical manifolds; see for example the related references in the survey paper~\cite{Lue1} and the discussion in~\cite{Sul}.) A complete affirmative answer to the Borel conjecture is known in dimensions less than four (see again~\cite{Lue1} for a survey). Moreover,
 by~\cite{BL,BLR}, the fundamental group of a circle bundle $M$ over a closed aspherical manifold $N$ with $\pi_1(N)$ hyperbolic  
 and $\dim(N)\neq 3,4$ satisfies the Farrell-Jones conjecture, and therefore the Borel conjecture, and so every homotopy equivalence of such a circle bundle is in fact homotopic to a homeomorphism. (See also~\cite{BHM} for self-maps of the base $N$.) 
\end{rem}

Beyond the Seifert volume for non-trivial circle bundles over hyperbolic surfaces~\cite{BG}, 
no other non-vanishing monotone invariant respecting the degree seemed to be known on higher dimensional circle bundles over aspherical manifolds with hyperbolic fundamental groups
(note that the simplicial volume vanishes as well ~\cite{Gromov}). A consequence of Theorem \ref{t:main} is that such a monotone invariant exists and it is given by the domination semi-norm. Recall that the domination semi-norm is defined by
\[
\nu_M(M'):=\sup\{|\deg(f)|\ | \ f\colon M'\longrightarrow M\},
\]
and it was introduced in~\cite{CL}. 
Theorem \ref{t:main} implies the following:

\begin{cor}
If $M$ is a not virtually trivial circle bundle over a closed oriented aspherical manifold with hyperbolic fundamental group, 
then $\nu_M(M)=1$.
\end{cor}

However, the domination semi-norm is not finite in general, because $M$ might admit maps of infinitely many different degrees from another manifold $M'$. 
Nevertheless, Theorem \ref{t:main} and the non-vanishing of the Seifert volume for non-trivial circle bundles over hyperbolic surfaces suggest the following: 

\begin{con}\label{conj}
In every dimension $n$, there is a homotopy $n$-manifold numerical invariant $I_n$ satisfying the inequality $I_n(M)\geq |\deg(f)|\cdot I_n(N)$ for each map $f\colon M\longrightarrow N$, which is positive and finite on every not virtually trivial circle bundle over a closed aspherical manifold with hyperbolic fundamental group. 
\end{con}

\section{Infinite sequences of coverings}

In this section, we reduce our discussion to self-coverings of a circle bundle over a closed oriented aspherical manifold with hyperbolic fundamental group and thus obtain a purely decreasing sequence of finite index subgroups of the fundamental group of this bundle. 

\subsection{Self-maps of aspherical manifolds with hyperbolic fundamental groups}\label{ss:hyperbolicmappings}

First, we observe that the hyperbolicity of the fundamental group of the base implies strong restrictions on the possible degrees of its self-maps:

\begin{prop}{\normalfont(\cite{BHM}).}\label{p:mappinghyperbolic}
Every self-map of non-zero degree of a closed aspherical manifold with non-elementary hyperbolic fundamental group is a homotopy equivalence. 
\end{prop}

\begin{proof}
There are two ways to see this. The first one (given in~\cite{BHM}) is purely algebraic, using the co-Hopf property of torsion-free, non-elementary hyperbolic groups~\cite{Sela1,Sela}. The other way uses  bounded cohomology and the simplicial volume; cf.~\cite{Min,Min1} and~\cite{Gromov1}.

Let $N$ be a closed oriented aspherical manifold whose fundamental group is non-elementary hyperbolic and $f\colon N\longrightarrow N$ be a map of non-zero degree. By~\cite{Sela1,Sela} (see also~\cite[Lemma 4.2]{BHM}), $\pi_1(N)$ is co-Hopfian (i.e. every injective endomorphism is an isomorphism), and so by the asphericity of $N$ it suffices to show that $f_*$ is injective. 
Suppose the contrary, and let a non-trivial element $x\in \ker(f_*)$. Since $f_*(\pi_1(N))$ has finite index in $\pi_1(N)$, 
there is some $n\in\N$ such that $x^n\in f_*(\pi_1(N))$, i.e. there is some $y\in\pi_1(N)$ such that $f_*(y)=x^n$. Clearly, $x^n\neq1$, because $\pi_1(N)$ is torsion-free, and so $y\notin\ker(f_*)$. Now, $f_*^2(y)=f_*(x^n)=1$, which means that $y\in\ker(f_*^2)$. By iterating this process, we obtain a purely increasing sequence 
\[
\ker(f_*)\subsetneq \ker(f_*^2)\subsetneq\cdots\subsetneq \ker(f_*^{m}) \subsetneq \ker(f_*^{{m+1}})\subsetneq\cdots.
\]
But the latter sequence contradicts Sela's result~\cite{Sela1,Sela} that for every endomorphism $\psi$ of a torsion-free hyperbolic group, there exists $m_0\in\N$ such that $\ker(\psi^k)=\ker(\psi^{m_0})$ for all $k\geq m_0$. We deduce that $f_*$ is injective, and therefore an isomorphism as required. 
 
 Alternatively to the above argument, since $\pi_1(M)$ is non-elementary hyperbolic, the comparison map from bounded cohomology to ordinary cohomology
\[
\psi_{\pi_1(M)}\colon H_b^n(\pi_1(M);\R)\longrightarrow H^n(\pi_1(M);\R)
\]
is surjective; cf.~\cite{Min,Min1,Gromov1}.  Thus, by the duality of the simplicial $\ell^1$-semi-norm and the bounded cohomology $\ell^\infty$-semi-norm (cf.~\cite{Gromov}), we deduce that $M$ has positive simplicial volume. This implies that every non-zero degree map $f\colon M\longrightarrow M$ has degree $\pm1$. In particular, $f$ is $\pi_1$-surjective, and thus an isomorphism, because $\pi_1(M)$ is Hopfian~\cite{Ma,Sela}.
\end{proof}

\subsection{Fundamental group and finite covers}\label{ss:cover} 

Let $M\stackrel{\pi}\longrightarrow N$ be an oriented circle bundle, where $N$ is a closed oriented aspherical manifold with $\pi_1(N)$ hyperbolic. We may assume that $\dim (N)\geq2$, otherwise we deal with the well-known case of $T^2$. 
 The fundamental group of $M$ fits into the central extension (cf.~\cite{Bo,CR})
\begin{equation*}\label{eq.fundamentalgroup}
1\longrightarrow C(\pi_1(M))\longrightarrow \pi_1(M)\stackrel{\pi_*}\longrightarrow\pi_1(N) \longrightarrow 1,
\end{equation*}
where $C(\pi_1(M))=\Z$ (note that $C(\pi_1(N))=1$, because $\pi_1(N)$ is torsion-free, non-elementary hyperbolic).

It is easy to observe that every finite covering of $M$ is of the same type. More precisely:

\begin{lem}{\normalfont(\cite[Lemma 4.6]{NeoIIPP}).}\label{l:cover}
Every finite cover $\overline{M}\stackrel{p}\longrightarrow M$ is a circle bundle $\overline{M}\stackrel{\overline{\pi}}\longrightarrow \overline{N}$, where $\overline{N}\stackrel{\overline{p}}\longrightarrow N$ is a finite covering.
\end{lem}

In particular, $p$ is a generalised bundle map covering $\overline{p}$ and the (infinite cyclic) center of $\pi_1(\overline{M})$ is mapped under $p_*$ into the center of $\pi_1(M)$.

\subsection{Reduction to fiberwise covering maps}\label{covers}

Now, let $f\colon M\longrightarrow M$ be a map of non-zero degree. 
We observe that $f$ is homotopic to a fiberwise covering map:

\begin{prop}\label{p:fibercover}
$f$ is homotopic to a fiberwise covering where the induced map $f_{S^1}\colon S^1\longrightarrow S^1$ has degree $\pm\deg(f)$.
\end{prop}
\begin{proof}
Let the composite map $\pi\circ f\colon M\longrightarrow N$ and the induced homomorphism
\[
(\pi\circ f)_*\colon \pi_1(M)\longrightarrow\pi_1(N).
\]
Since the center of $\pi_1(N)$ is trivial, we derive, after lifting $f$ to a $\pi_1$-surjective map $\overline{f}\colon M\longrightarrow \overline{M}$ (where $\overline{M}\stackrel{p}\longrightarrow M$ corresponds to $f_*(\pi_1(M))$), that the center of $\pi_1(M)$ is mapped under $(\pi\circ f)_*$ to the trivial element of $\pi_1(N)$; see Lemma \ref{l:cover} and the lines above that. Thus $f$ factors up to homotopy through a self-map $g\colon N\longrightarrow N$, i.e. $\pi\circ f =g\circ\pi$ (up to homotopy).

Clearly, $\deg(g)\neq 0$, otherwise $f$ would factor through the degree zero map from the pull-back bundle of $g$ along $\pi$ to $M$, which is impossible because $\deg(f)\neq 0$. Now, Proposition \ref{p:mappinghyperbolic} implies that $g$ is a homotopy equivalence of $N$ (in particular $\deg(g)=\pm 1$). Hence, the induced map $f_{S^1}$ on the $S^1$ fiber is homotopic to a self-covering of degree
\[
\deg(f_{S^1})=\pm\deg(f).
\]
\end{proof}

Since every map of degree $\pm1$ is $\pi_1$-surjective, every self-map of $M$ of degree $\pm1$ is a homotopy equivalence, answering thus in the affirmative Problem \ref{Hopfprob}. More interestingly, the above proposition gives the following strong affirmative answer to Problem \ref{Hopfstrong}; cf. Remark \ref{r:homeo}:

\begin{cor}\label{c:Hopfprob}
Let $M$ be an oriented circle bundle over a closed oriented aspherical manifold $N$ with hyperbolic fundamental group and $\dim(N)\neq3,4$.  Every self-map of $M$ of non-zero degree is either homotopic to a homeomorphism or homotopic to a non-trivial covering.
\end{cor}

Consider now the iterates
\[
f^m\colon M\longrightarrow M, \ m\geq1.
\]
By Proposition \ref{p:fibercover}, each $f^m$ is homotopic to a fiberwise covering of degree 
\[
(\deg(f))^m=[\pi_1(M):f^m_*(\pi_1(M))],
\]
i.e. for each $m$, the homomorphism
\[
f^m_*\colon\pi_1(M)\longrightarrow\pi_1(M) 
\]
maps every element $x\in C(\pi_1(M))=\Z=\langle z\rangle$ to $x^{\pm\deg(f^m)}\in C(\pi_1(M))$ and induces an isomorphism on $\pi_1(N)=\pi_*(\pi_1(M))$. In particular, when $\deg(f)>1$, we obtain the following:
\begin{cor}\label{c:sequence}
If $f\colon M\longrightarrow M$ has degree greater than one, then 
there is a purely decreasing infinite sequence of subgroups of $\pi_1(M)$ given by
\begin{equation}\label{eq.sequence}
\pi_1(M)\supsetneq f_*(\pi_1(M))\supsetneq\cdots\supsetneq f^m_*(\pi_1(M))\supsetneq f^{m+1}_*(\pi_1(M))\supsetneq\cdots.  
\end{equation}
\end{cor}

\section{Distinguishing between trivial and non-trivial bundle}

Now we will show that the existence of sequence (\ref{eq.sequence}) implies that $\pi_1(M)$ has a finite index subgroup which is isomorphic to a direct product and thus $M$ is virtually trivial. To this end, we will construct a presentation of $\pi_1(M)$ by a product of two infinite index subgroups.

\subsection{Groups infinite index presentable by products}

An infinite group $\Gamma$ is said to be {\em infinite index presentable by products} or {\em IIPP} if there exist two infinite subgroups $\Gamma_1,\Gamma_2\subset\Gamma$ that commute elementwise, such that $[\Gamma:\Gamma_i]=\infty$ for both $\Gamma_i$ and the multiplication homomorphism
\begin{equation*}\label{eq.IIPP}
\Gamma_1\times\Gamma_2\longrightarrow\Gamma
\end{equation*}
surjects onto a finite index subgroup of $\Gamma$.

The notion of groups IIPP was introduced in~\cite{NeoIIPP} in the study of maps of non-zero degree from direct products to aspherical manifolds with non-trivial center. The concept of groups presentable by products (i.e. without the constraint on the index) was introduced in~\cite{KL}. It is clear that when $\Gamma$ is a {\em reducible} group, that is, virtually a product of two infinite groups, then $\Gamma$ is IIPP. Thus, a natural problem is to determine when these two properties are equivalent. In general, they are not equivalent as shown in~\cite[Section 8]{NeoIIPP}, however their equivalence is achieved under certain assumptions:

\begin{thm}{\normalfont(\cite[Theorem D]{NeoIIPP}).}\label{t:IIPPequiv}
Suppose $\Gamma$ fits into a central extension
\[
1\longrightarrow C(\Gamma)\longrightarrow\Gamma\longrightarrow\Gamma/C(\Gamma)\longrightarrow1,
\]
where $\Gamma/C(\Gamma)$ is not presentable by products. Then $\Gamma$ is IIPP if and only if it is reducible.
\end{thm}

The following theorem characterizes aspherical circle bundles, when the fundamental group of the base is not presentable by products:

\begin{thm}{\normalfont(\cite[Theorem C]{NeoIIPP}).}\label{t:char}
 Let $M \stackrel{\pi}\longrightarrow N$ be a circle bundle over a closed aspherical manifold $N$ whose fundamental group $\pi_1(N)$ is not presentable by
products. Then the following are equivalent:
 \begin{itemize}
 \item[(1)] $M$ admits a map of non-zero degree from a direct product;
 \item[(2)] $M$ is finitely covered by a product $S^1 \times \overline{N}$, for some finite cover $\overline{N} \longrightarrow N$;
 \item[(3)] $\pi_1(M)$ is reducible;
 \item[(4)] $\pi_1(M)$ is IIPP.
 \end{itemize}
\end{thm}

Since non-elementary hyperbolic groups are not presentable by products~\cite{KL}, each circle bundle $M$ over a closed aspherical manifold $N$ with $\pi_1(N)$ hyperbolic fulfills the assumptions of Theorems \ref{t:IIPPequiv} and \ref{t:char}. Using this, we will be able to deduce that $M$ is virtually trivial. Furthermore, our presentation by products for $\pi_1(M)$ will have trivial kernel; 
see Remark \ref{r:notPP}. 

\subsection{An infinite index presentation by products}

Under the assumption of the existence of $f^m\colon M\longrightarrow M$ with $\deg(f^m)=(\deg(f))^m>1$ for all $m\geq 1$, and thus of sequence (\ref{eq.sequence}), we consider the subgroup of $\pi_1(M)$ defined by
\begin{equation*}
G:=\mathop{\cap}_{m}f^m_*(\pi_1(M)).
\end{equation*}

First, we observe the general fact (i.e. without using the specific form of each $f^m_*(\pi_1(M))$) that $G$ has infinite index in $\pi_1(M)$.
Let us suppose the contrary, i.e. that $[\pi_1(M):G]<\infty$. Then by 
\[
[\pi_1(M):f^m_*(\pi_1(M))]\leq[\pi_1(M):G]
\]
for all $m$, and the fact that $\pi_1(M)$ contains only finitely many subgroups of a fixed index, we deduce that there exists $n$ such that $f^n_*(\pi_1(M))=f^k_*(\pi_1(M))$ for all $k\geq n$. This is however impossible by Corollary \ref{c:sequence}. 
Now, we will show that $\pi_1(M)$ admits a presentation by the product $C(\pi_1(M))\times G$. Let
\begin{equation}\label{eq.presentation}
\varphi\colon C(\pi_1(M))\times G\longrightarrow \pi_1(M)
\end{equation}
be the multiplication map. Since each element of $C(\pi_1(M))$ commutes with every element of $G$, we deduce that $\varphi$ is in fact a well-defined homomorphism. 

We claim that $\varphi$ surjects onto a finite index subgroup of $\pi_1(M)$, i.e. that $C(\pi_1(M))G$ has finite index in $\pi_1(M)$. To this end, we will use the specific description of $f^m$ and $f^m_*(\pi_1(M))$. 
In Section \ref{covers}, we have seen that for every $m$, the composite $f^m$ is a fiberwise covering of degree $\deg(f^m)$ on the fibers that induces an isomorphism on $\pi_1(N)$, and even more it induces a homotopy equivalence of $N$. In particular, for every $m\geq1$ we obtain a short exact sequence
\[
1\longrightarrow\langle z^{\deg(f)^m}\rangle\longrightarrow f^m_*(\pi_1(M))\stackrel{}\longrightarrow\pi_1(N_m)\longrightarrow1,
\]
where $\pi_1(N_m)\cong\pi_1(N)$. Hence, $\pi_1(M)/f^m_*(\pi_1(M))\cong\Z/\deg(f)^m\Z$, for all $m\geq1$, and so $\pi_1(M)/G\cong\Z$. Thus, we obtain a short exact sequence (induced by $\pi_*$)
\[
1\longrightarrow(C(\pi_1(M))G)/G\longrightarrow\pi_1(M)/G\stackrel{\overline{\pi}_*}\longrightarrow\pi_1(N)/\pi_*(G)\longrightarrow1.
\]
Since $(C(\pi_1(M))G)/G\cong\Z$, we conclude that $\pi_*(G)$ is a finite index subgroup of $\pi_1(N)$.

Let now $x\in\pi_1(M)$. If $x=z^s\in C(\pi_1(M))=\langle z\rangle$, then $\varphi(x,1)=x$. If $x\notin C(\pi_1(M))$, then $\pi_*(x)$ is not trivial in $\pi_1(N)$ and so there exists $t\geq 0$ such that $\pi_*(x^t)\in\pi_*(G)$, i.e. $\pi_*(x^t)=\pi_*(g)$ for some $g\in G$. Thus $x^t=z^{a}g$ for some $a\in\Z$, and so $\varphi(z^a,g)=x^t$. We conclude that $\varphi(C(\pi_1(M))\times G)=C(\pi_1(M))G$ has finite index in $\pi_1(M)$.

Since moreover $C(\pi_1(M))$ and $G$ have infinite index in $\pi_1(M)$, we conclude that the presentation given in (\ref{eq.presentation}) is an infinite index presentation by products. Theorem \ref{t:char} implies that $\pi_1(M)$ is reducible and $M$ is virtually a trivial circle bundle. 

This finishes the proof of Theorem \ref{t:main}.

\begin{rem}\label{r:notPP}
The kernel of $\varphi$ must be trivial, because it is isomorphic to $C(\pi_1(M))\cap G$ which is trivial. 
Thus $C(\pi_1(M))G$ is isomorphic to the fundamental group of a trivial circle bundle that covers $M$. In particular, the property of $\pi_1(N)$ that is not presentable by products was not necessary for our proof.

An alternative way to see that $C(\pi_1(M))\cap G$ is trivial is to observe that
\[
[C(\pi_1(M)):C(\pi_1(M))\cap G]=[\pi_1(M):G]=\infty.
\]
Together with the fact that $C(\pi_1(M))=\Z$, we conclude that $C(\pi_1(M))\cap G$ is trivial. 
\end{rem}

The proof of Corollary \ref{c:Hopf} is now straightforward:

\begin{proof}[Proof of Corollary \ref{c:Hopf}]
Let $M$ be a circle bundle over a closed oriented aspherical manifold $N$ with $\pi_1(N)$ hyperbolic and $f\colon M\longrightarrow M$ be a map of non-zero degree. As we have seen in Section \ref{covers}, if $\deg(f)=\pm1$, then $f$ is a homotopy equivalence and, if $\deg(f)\neq\pm1$, then $f$ is homotopic to a non-trivial fiberwise covering (and to a non-trivial covering when $\dim(N)\neq3,4$; see Remark \ref{r:homeo}). In the latter case, Theorem \ref{t:main} implies moreover that $M$ is virtually $S^1\times \overline N$ for some finite covering $\overline N\to N$. 
\end{proof}

\begin{rem}
When $M$ has torsion Euler class $e\in H^2(N)$, we have seen in Example \ref{e:torsion} that $M$ admits a self-map $f$ of degree greater than one. Recall that a product finite covering  $S^1\times \overline N \longrightarrow M$ is obtained by pulling back  $M\stackrel{\pi}\longrightarrow N$ along the finite covering $\overline N\to N$ that corresponds to the finite index subgroup
\[
H:=\ker(\pi_1(N)\stackrel{h}\longrightarrow H_1(N)\stackrel{\pi_T}\longrightarrow  \mathrm{Tor} H_1(N)) \subseteq \pi_1(N),
\]
where $h$ denotes the Hurewicz map and $\pi_T$ is the projection to the torsion of $H_1(N)$. (Note that $e$ lies in $\mathrm{Tor} H_1(M)$ by the Universal Coefficient Theorem.)
The groups $H$ and $\pi_*(G)$ are commensurable in $\pi_1(N)$ because
\[
[\pi_*(G):\pi_*(G)\cap H]\leq[\pi_1(N):H]<\infty \ \text{and} \ [H:\pi_*(G)\cap H]\leq[\pi_1(N):\pi_*(G)]<\infty.
\]
\end{rem}

\section{The Hopf property and the strong version of Hopf's problem}\label{s:Hopfdiscussion}

In this section we discuss the Hopf property for circle bundles over aspherical manifolds with hyperbolic fundamental groups and Problem \ref{Hopfstrong} more generally. 

\subsection{The Hopf property}

First, we show that the fundamental groups of circle bundles over aspherical manifolds with hyperbolic fundamental groups are Hopfian: 

\begin{prop}\label{p:Hopfproperty1}
If $M$ is a circle bundle over a closed oriented aspherical manifold with hyperbolic fundamental group, then every finite index subgroup of $\pi_1(M)$ is Hopfian.
\end{prop}
\begin{proof}
Let $M\stackrel{\pi}\longrightarrow N$ be a circle bundle, where $N$ is a closed oriented aspherical manifold with $\pi_1(N)$ hyperbolic. (As before, we can assume that $\pi_1(N)$ is not cyclic.) Since every finite covering of $M$ is of the same type (cf. Lemma \ref{l:cover}), it suffices to show that $\pi_1(M)$ is Hopfian. 

Let $\phi\colon\pi_1(M)\longrightarrow\pi_1(M)$ be a surjective homomorphism. Then $\phi(C(\pi_1(M)))\subseteq C(\pi_1(M))$, and so the composite homomorphism
$\pi_*\circ\phi\colon \pi_1(M)\longrightarrow\pi_1(N)$ maps $C(\pi_1(M))$ to the trivial element of $\pi_1(N)$. In particular, there exists a surjective homomorphism
$\overline{\phi}\colon\pi_1(N)\longrightarrow \pi_1(N)$ such that $\overline{\phi}\circ\pi_*=\pi_*\circ\phi$. Now $\overline{\phi}$ is injective as well (and so an isomorphism), because $\pi_1(N)$ is Hopfian, being hyperbolic and torsion-free~\cite{Ma,Sela}. 
Then, using again the surjectivity of $\phi$, we deduce that 
\[
\phi\vert_{C(\pi_1(M))}\colon C(\pi_1(M))\longrightarrow C(\pi_1(M))
\]
is also surjective. Since $C(\pi_1(M))=\Z$ is Hopfian, we conclude that $\phi\vert_{C(\pi_1(M))}$ is in fact an isomorphism. 
\begin{figure}
     \[
\xymatrix{
1 \ar[r]^{} &  C(\pi_1(M)) \ar[d]^{\phi\vert_{C(\pi_1(M))}} \ar[r]^{} &   \pi_1(M) \ar[d]^{\phi} \ar[r]^{\pi_*} & \pi_1(N)\ar[d]^{\overline{\phi}} \ar[r]^{} & 1 \\
1 \ar[r]^{} &C(\pi_1(M)) \ar[r]^{}  &    \pi_1(M)\ar[r]^{\pi_*} &  \pi_1(N) \ar[r]^{} & 1\\
}
    \]
\caption{\small The Hopf property for $\pi_1(M)$.}
\label{f:Hopf}
 \end{figure}
Now, the five-lemma for the commutative diagram in Figure \ref{f:Hopf} implies that $\phi$ is an isomorphism as well.
\end{proof}
 
In this way, we obtain also an alternative proof of the fact that every self-map of $M$ of degree $\pm 1$ is a homotopy equivalence. Of course, the above group theoretic argument uses the same line of argument as the proof of Proposition \ref{p:fibercover}, with the difference that it starts with a stronger assumption, namely that $\phi$ is surjective.

\subsection{Infinite decreasing sequences and Problem \ref{Hopfstrong}}

The fact that every finite index subgroup of the fundamental group of a circle bundle over an aspherical manifold $N$ with hyperbolic $\pi_1(N)$ has the Hopf property is actually conjectured to be true for all aspherical manifolds. Beyond that this would immediately verify Problem \ref{Hopfprob} for every aspherical manifold, it also gives evidence for an affirmative answer to Problem \ref{Hopfstrong}. Namely, let $f\colon M\longrightarrow M$ be a map of degree $\deg(f)>1$ and suppose that every finite index subgroup of $\pi_1(M)$ is Hopfian. Then, as in the case of non-trivial coverings, there is a purely decreasing infinite sequence 
\begin{equation*}
\pi_1(M)\supsetneq f_*(\pi_1(M))\supsetneq\cdots\supsetneq f^m_*(\pi_1(M))\supsetneq f^{m+1}_*(\pi_1(M))\supsetneq\cdots.  
\end{equation*}
The proof of this claim can be found along the lines of the proof of Theorem 14.40 of~\cite{Lue}, but let us give the details for completeness:
Suppose the contrary, i.e. that there is some $n$ such that $f^n_*(\pi_1(M))=f^k_*(\pi_1(M))$ for all $k\geq n$. Let $M_n\stackrel{p_n}\longrightarrow M$ be the finite covering of $M$ corresponding to $f^n_*(\pi_1(M))$ and denote by $\overline{f^n}\colon M\longrightarrow M_n$ the lift of $f^n$, which induces a surjection on $\pi_1$. Since $f^n_*(\pi_1(M))=f^{2n}_*(\pi_1(M))$, we deduce that the composite map
$\overline{f^n}\circ p_n\colon M_n\longrightarrow M_n$ induces a surjection
\[
(\overline{f^n}\circ p_n)_*\colon \pi_1(M_n)\longrightarrow \pi_1(M_n).
\]
Since $\pi_1(M_n)$ is Hopfian, we deduce that $(\overline{f^n}\circ p_n)_*$ is an isomorphism, and so a homotopy equivalence, because $M_n$ is aspherical. In particular,  we obtain
\[
\deg(\overline{f^n}),\deg(p_n)\in\{\pm 1\}, 
\]
which leads to the absurd conclusion that $\deg(f)=\pm 1$. 

\bibliographystyle{amsplain}

\end{document}